\documentclass[12pt]{amsart}
\usepackage{graphicx,color,subcaption,enumerate}
\usepackage{amssymb}
\usepackage{amscd}
\usepackage[all]{xy}
\usepackage[T1]{fontenc}
\usepackage{relsize}
\usepackage{tgheros}
\usepackage{geometry}

\newcommand{\Real}{\mathbb{R}}

\newcommand{\M}{\mathbb{M}}

\newtheorem{theorem}{Theorem}[section]
\newtheorem{proposition}{Proposition}[section]
\newtheorem{lemma}{Lemma}[section]
\newtheorem{corollary}{Corollary}[section]
\newtheorem{notation}{Notation}[section]
\newtheorem*{main}{Main Theorem}

\newcounter{exa}
\newtheorem{example}[exa]{Example}

\newtheorem{remark}{Remark}[section]

\theoremstyle{definition}
\newtheorem{definition}{Definition}[section]

\usepackage{hyperref}

\hypersetup{
    colorlinks = true,
    linkcolor = blue
}


\title[Random circular billiards on surfaces of constant curvature]{Random circular billiards on surfaces of constant curvature: Pseudo integrability and mixing}

\author{T\'ulio Vales}
\email{tuliovales@ufu.br}

\date{\today}

\begin{document}

\begin{abstract}
Given a random map $(T_1, T_2, T_3, T_4, p_1, p_2, p_3, p_4)$, we define a random billiard map on a surface of constant curvature (Euclidean plane, hyperbolic plane, or the sphere). The Liouville measure is invariant for this billiard map. Finally, we show some dynamical properties such as ergodicity in the case of random circular billiards.
\end{abstract}

\maketitle
\section{Introduction}
Considering a surface $\M$ and a closed curve $\Gamma\subset\M$, the (deterministic) \textit{billiards} consists of a particle's free motion in the region bounded by $\Gamma$.
Such movement is realized through geodesics, and collisions with $\Gamma$ follow the law of reflection: angle of incidence equals angle of reflection.
In the present paper we consider billiards on (simply connected) surfaces of constant curvature: the Euclidean plane $\Real^2$, the hyperbolic plane $\mathbb{H}^2$, or the sphere $\mathbb{S}^2$.

In general, there are no explicit expressions for the billiard map, but in the case of smooth curves the expression of its derivatives as a function of the curvature $k$, the distance between collisions, and the sines of the reflected angles are well known.
For example, if $\Gamma$ is a simple, closed, regular, smooth, oriented curve parametrized by arc length in the Euclidean plane $\Real^2$, and if $F$ is the corresponding billiard map, then
\begin{equation*}
DF(s_0,\theta_0) = \frac{1}{\sin \theta_1} \begin{bmatrix} l_{01} k_0- \sin \theta_0 & l_{01} \\ k_1 (l_{01} k_0- \sin \theta_0) -k_0 \sin \theta_1 & l_{01} k_1- \sin \theta_1 \end{bmatrix},
\end{equation*}
where $(s_1,\theta_1)=F(s_0,\theta_0)$; $k_0$ and $k_1$ are the curvature of $\Gamma$ at points $\Gamma(s_0)$, $\Gamma(s_1)$, respectively; and $l_{01}$ is the distance between collisions.
For more details, see \cite{markarian}.

In the case of oval curves on a surface of constant curvature, the derivative of the corresponding billiard map can be found in \cite{luciano}.
For the reader interested in others billiards on a surface of constant curvature, we recommend \cite{gutkin1999hyperbolic}.

It is already known that every circular billiards on a surface of constant curvature has the same behavior.
That is, the trajectories are periodic if the exit angle $\theta_0$ is rational with $\pi$, and dense if $\theta_0$ is irrational with $\pi$.
Besides, this billiard map is integrable.
In particular, integrability makes them non-ergodic (with respect to the Liouville measure) independently of whether the exit angle is rational or irrational with $\pi$.

In this paper we define a random billiard map in which the angles of reflection will be ``drawn'' according to a probability distribution as in \cite{Vales}.
Considering this map in a general curve, it can be shown the invariance of the Liouville measure.

The motivation for the study of random billiards comes from the theory of gas kinetics. Through various measurements, it was observed in some studies that injecting a certain quantity of gases into an infinite test tube, the fraction of particles that exit after many collisions in a particular direction is proportional to the cosine of the initial angle. This behavior is referred to as the Knudsen’s law (see \cite{feres} for more details). This being said, we have two explanations for this phenomenon: the first one is that the infinite test tube is not completely smooth (with deformations); the other is that collisions between particles contribute to the reflected angle. To endorse the first one, if the test tube were completely smooth, the angle of incidence of a particle would be equal to the angle of reflection, as one can see in Figure~\ref{tubo de ensaio}.

\begin{figure}[!htb]
\centering
\includegraphics[scale=0.3]{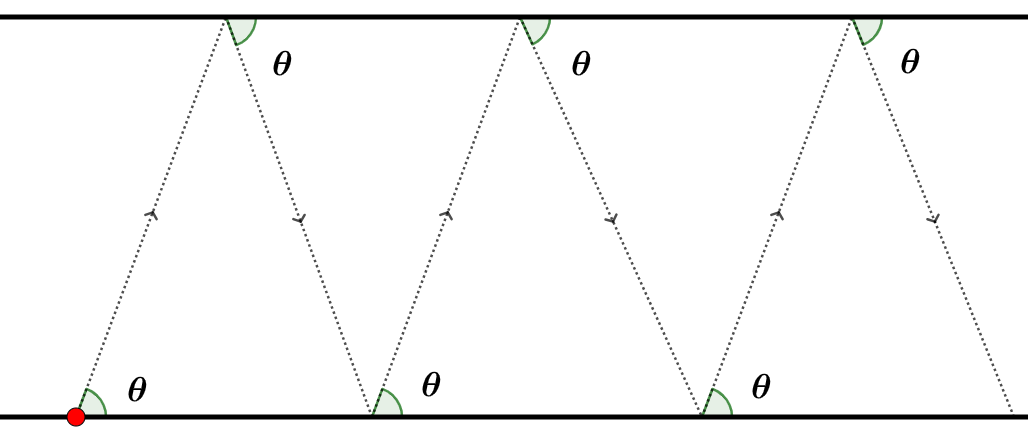}
\caption{Collision of a particle in a two-dimensional infinite test tube with no deformations.}\label{tubo de ensaio}
\end{figure}

In order to illustrate a random billiard, we can imagine the curve $\Gamma$ having minor irregularities (deformations) that will be modeled by isosceles triangles with base angle $\alpha$.
These irregularities are microscopic, meaning that the region bounded by $\Gamma$ is relatively large compared to the triangular irregularity.
The particular case of a circle as the billiard curve is depicted in Figure~\ref{triangular irregularidade} below.

\begin{figure}[!htb]
\centering
\includegraphics[scale=0.6]{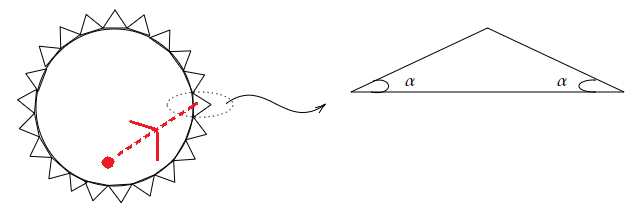}
\includegraphics[scale=0.6]{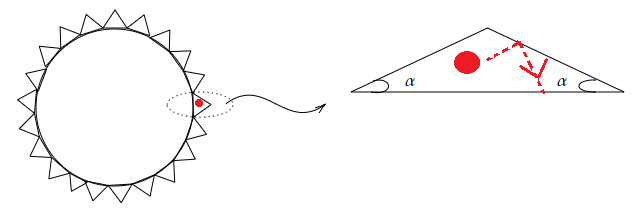}
\caption{Circular billiards with triangular irregularities.}
\label{triangular irregularidade}
\end{figure}


Since we are dealing with a non-deterministic dynamical system, every starting point and initial angle will not have a single trajectory.
Due to this we will need the aid of the \textit{transition probability kernel}, denoted hereafter as $K(\theta, A)$, which describes the possible trajectories of particles reaching a set $A\subset\M$ with $\theta$ as the angle of incidence (initial angle). Furthermore, for a dynamical study of this system it is required an invariant measure which will be provided through such operator $K(\theta,A)$.


When the curve $\Gamma$ is a circle we will prove the Lyapunov exponent of the random billiard map is null for every point and direction.
In addition, we will show the abundance of dense trajectories in this particular case.
We will also see that random circular billiard map has a certain integrability.
However, in contrast to what happens in the deterministic case, circular random billiards can be mixing (recall the parameter $\alpha$ that comes from the irregularities):

\begin{main}
Considering a surface of constant curvature, the random circular billiard map is mixing if and only if $\frac{\alpha}{\pi}$ is an irrational number.
\end{main}

Let us present the division of this work in the sequence.
Section 2 contains the formal definition of circular billiards on surfaces of constant curvature, and its explicit expression.
In Sections 3 and 4 we will introduce the Feres random map and determine how it will define random billiards.
Section 5 deals with the particular case of random circular billiards.
Finally, in Section 6 we prove the Main Theorem and also the pseudo integrability of random circular billiards.

\section{Circular billiards on surface of constant curvature}

Let us begin with some basic definitions.
The sphere $\mathbb{S}^2$ is the set $\mathbb{S}^2 = \{(x,y,z): x^2+y^2+z^2=1\}$.
Observe that the function $\varphi_{\mathbb{S}^2}(\rho, \theta) = (\sin \rho \cos \theta, \sin \rho \sin \theta, \cos \rho )$ with $0 < \rho < \pi$ and $0 < \theta < 2\pi$ transforms $\mathbb{S}^2$ into a surface of $\mathbb{R}^3$.
The hyperbolic plane $\mathbb{H}^2$ is a leaf of the hyperboloid, that is, the set $\mathbb{H}^2 = \{(x,y,z): x^2+y^2-z^2=-1 \mbox { and } z>0\}$.
Again, notice that the function $\varphi_{\mathbb{H}^2}(\rho, \theta) = (\sinh \rho \cos \theta, \sinh \rho \sin \theta, \cosh \rho)$  with $\rho>0$ and $0 < \theta < 2\pi$ transforms $\mathbb{H}^2$ into a surface of $\mathbb{R}^3$.

By the Killing-Hopf Theorem (\textit{cf.} \cite{stillwell2012geometry}), every two-dimensional simply-connected space with constant curvature $k$ is isometric to:
\begin{enumerate}
    \item the Euclidean plane $\mathbb{R}^2$ if $k = 0$;
    \item  the hyperbolic plane $\mathbb{H}^2$ if $k=-1$;
    \item the sphere $\mathbb{S}^2$ if $k=1$.
\end{enumerate}
Consequently, $\mathbb{R}^2$, $\mathbb{S}^2$ and $\mathbb{H}^2$ are essentially the only complete, simply-connected Riemannian manifolds with constant curvature.
Therefore, it is sufficient  to consider these surfaces.

In order to determine the billiards map in such surface, let us consider a closed convex curve $\Gamma$ parametrized by arc length as a function $\lambda: [0,L)\to \Gamma$.
Collisions with $\Gamma$ follows the \textit{law of reflection}: the angle of incidence $\theta_1$ equals the angle of reflection $\theta_2$.

Given a point $\lambda(t_1)=S_1 \in \Gamma$ and a direction $\theta\in [0,\pi)$, the angle of incidence $\theta_1(S_1,\theta)$ reaching the point $\lambda(t_2)=S_2$ is the one between $\lambda'(t_2)$ and the geodesic from $S_1$ into $S_2$.
Therefore, we define the billiard map as
\begin{equation*}
\begin{matrix}
F\colon & \Gamma\times[0,\pi) & \longrightarrow & \Gamma\times[0,\pi)\\
& (S_1,\theta) & \longmapsto & (S_2(S_1,\theta), \theta_1(S_1, \theta)).
\end{matrix}
\end{equation*}
Observe that $\theta_2=\theta_1(S_1,\theta)$ due to the law of reflection.

Let us consider now the particular case when the curve $\Gamma$ is a circle with radius $r_0>0$ in $\Real^2$, $\mathbb{H}^2$, or $\mathbb{S}^2$.
In the following we will make use of the constant $h\in\Real$ given by
\begin{equation*}
h= \begin{cases}
r_0 &  \text{ in }\mathbb{R}^2 \\
\sinh(r_0) & \text{ in }\mathbb{H}^2 \\
\sin (r_0) & \text{ in }\mathbb{S}^2.
\end{cases}
\end{equation*}
Taking into account last paragraph's notation, it is known that the arc length from $S_1$ to $S_2$ is $l(\gamma)=h\gamma$, where $\gamma>0$ is the angle between those points.
Thus, we can say that $S_2=S_1+l(\gamma)$.
We would like to mention that $\gamma$ depends only on $S_1$ and $\theta$.
This case is illustrated in Figure~\ref{circularbilhar}.

\begin{figure}[h!]
    \centering
    \includegraphics[scale=0.8]{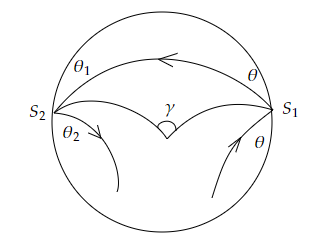}
    \caption{Circular billiards in $\mathbb{R}^2$, $\mathbb{H}^2$, or $\mathbb{S}^2$.}
    \label{circularbilhar}
\end{figure}

We can then give the explicit expression of the circular billiards map in $\mathbb{R}^2, \mathbb{H}^2$ or $\mathbb{S}^2$.
Indeed, it is the map $F\colon \Gamma\times[0,\pi)\to \Gamma\times[0,\pi)$ defined by
\begin{equation}\label{bilharcircular}
    F(S,\theta) = (S+l(\gamma(\theta)) \mod 2\pi h, \theta).
\end{equation}
Also, its derivative is given by
\begin{equation*}
DF(S,\theta) = \begin{bmatrix} 1 & \gamma'(\theta)h \\
0 & 1 \end{bmatrix}.
\end{equation*}

Notice that the action of the first coordinate of $F$ in the circle $\Gamma$ can be seen as a rotation.
That is, a particle at position $S\in\Gamma$ moves to position $S+l(\gamma(\theta)) \mod 2\pi h$ and, since the second coordinate of $F$ is constant, this movement is indeed a rotation.
Hence, if $\gamma(\theta) = \frac{2m\pi}{n}$, then $F^n(S,\theta)=(S,\theta)$ and therefore the orbit of the particle is periodic.
Otherwise, if $\gamma(\theta) = \beta \pi$, where $\beta$ is an irrational number, then the orbit of the particle is dense at the boundary of the circular table.

We also have circular billiards that are integrable, that is, exists an $F$-invariant map $G\colon\Gamma\times[0,\pi)\to\Gamma\times[0,\pi)$ that is not constant.
Indeed, if one takes $G(S,\theta)=\theta$, then $G(F(S,\theta)) = G(S + 2\theta \mod 2\pi h, \theta) = \theta = G(S,\theta)$, as we claimed.

\begin{notation}
From now on we will consider $\mathbb{M} = \mathbb{R}^2$, $\mathbb{H}^2$, or $\mathbb{S}^2$, and $M=\Gamma\times[0,\pi)$ as the rectangle where the circular billiard map $F$ is defined.
\end{notation}

\section{Random Maps and Markov Chains}
Let $(X, \mathcal{B}(X))$ be a measurable space.
If we consider maps $T_i\colon X \to X$ and functions $p_i\colon X \to [0,1]$ with $i = \{1, 2, \dots, N\}$, then we define the \textit{random map} $T\colon X\to X$ as $T(\theta)= T_i(\theta)$ with probability $p_i(\theta)$.
Iteratively, for each $n \in \mathbb{N}$, we have $T^{(n)}(\theta) = T_{i_n}\circ T_{i_{n-1}} \circ \dots \circ T_{i_1}(\theta)$ with probability $p_{i_1}(\theta)p_{i_2}(T_{i_1}(\theta))\dots p_{i_n}(T_{i_{n-1}}\circ \dots \circ T_{i_1}(\theta))$.

The \textit{transition probability kernel} of the random map $T$ is given by 
\begin{equation}\label{nucleo}
K(\theta, A) = \sum_{i=1}^{N} p_i(\theta) 1_{A}(T_i(\theta))\qquad\text{for all $\theta\in X$ and $A\in\mathcal{B}(X)$}.
\end{equation}
This operator describes the evolution of an initial distribution (probability measure) $\nu$ on $(X, \mathcal{B}(X))$ under the random map $T$ iteratively as: $\nu^{(0)}= \nu$ and
\begin{equation}\label{def:measure evolution}
\nu^{(n+1)}(A) = \int_X K(\theta,A)d\nu^{(n)}(\theta)\qquad\text{for all $A\in \mathcal{B}(X)$ and $n\ge 1$}.
\end{equation}

The next definition deals with a sufficient condition on measures to be invariant by the random map $T$.
\begin{definition}\label{medida invariante}
We say that a measure $\mu$ in $(X,\mathcal{B}(X))$ is $T$-invariant if $\mu(A) = \sum_{i=1}^{N} \int_{T^{-1}_i(A)} p_i(\theta)d\mu(\theta)$ for all $A \in \mathcal{B}(X)$. 
\end{definition}

In this paper we will consider the following particular case of random map introduced in \cite{feres}.
For this, let $X=[0,\pi]$ endowed with its Borel $\sigma$-algebra, and fix an $\alpha<\dfrac{\pi}{6}$.
Remember that this parameter refers to deformations in the circle $\Gamma$, as shown in Figure~\ref{triangular irregularidade}.

The \textit{Feres random map} $T: [0,\pi] \to [0,\pi]$ is given by 
\begin{equation}\label{M1}
T(\theta) = T_i(\theta) \text{ with probability } p_i(\theta)
\end{equation}
where $T_1(\theta) = \theta + 2\alpha$, $T_2(\theta) = -\theta + 2\pi - 4\alpha$, $T_3(\theta) = \theta -2\alpha$ and $T_4(\theta) = -\theta + 4\alpha$, and their respective probabilities are
\begin{align*}
p_1(\theta) &= \begin{cases} 1, & \hspace{26.5mm}\text{if } \theta \in [0,\alpha),\\
u_{\alpha}(\theta), & \hspace{26.5mm}\text{if } \theta \in [\alpha, \pi -3\alpha),\\
2\cos(2\alpha)u_{2\alpha}(\theta), & \hspace{26.5mm}\text{if } \theta \in [\pi - 3\alpha, \pi -2\alpha),\\
0, & \hspace{26.5mm}\text{if } \theta \in [\pi-2\alpha, \pi],\end{cases}\\
p_2(\theta) &= \begin{cases} 0,  & \hspace{11.5mm}\text{if } \theta \in [0,\pi -3\alpha),\\
u_{\alpha}(\theta)-2\cos (2\alpha)u_{2\alpha}(\theta),  & \hspace{11.5mm}\text{if } \theta \in [\pi -3\alpha, \pi -2\alpha),\\
u_{\alpha}(\theta),  & \hspace{11.5mm}\text{if } \theta \in [\pi - 2\alpha, \pi -\alpha),\\
0, & \hspace{11.5mm}\text{if } \theta \in [\pi-\alpha, \pi],
\end{cases}\\
p_3(\theta) &= \begin{cases} 0, & \hspace{23mm}\text{if } \theta \in [0,2\alpha),\\
2\cos (2\alpha)u_{2\alpha}(-\theta), & \hspace{23mm}\text{if } \theta \in [2\alpha, 3\alpha),\\
u_{\alpha}(-\theta), & \hspace{23mm}\text{if } \theta \in [3\alpha, \pi -\alpha),\\
1, & \hspace{23mm}\text{if } \theta \in [\pi-\alpha, \pi],
\end{cases}\\
p_4(\theta) &= \begin{cases} 0, & \hspace{5mm}\text{if } \theta \in [0,\alpha),\\
u_{\alpha}(-\theta), & \hspace{5mm}\text{if } \theta \in [\alpha, 2\alpha),\\
u_{\alpha}(-\theta)-2\cos(2\alpha)u_{2\alpha}(-\theta), & \hspace{5mm}\text{if } \theta \in [2\alpha, 3\alpha),\\
0, & \hspace{5mm}\text{if } \theta \in [3\alpha, \pi]  \end{cases}
\end{align*}
with $u_{\alpha}(\theta) = (1+\frac{tg\alpha}{tg\theta}).$
For more details see \cite{lamb,feres, Vales}.

The result below is a particular case of \cite[Proposition~2.1]{feres}.
We present its proof for the sake of completeness.

\begin{proposition}\label{prop:T-invariance}
If $T$ is the Feres random map as in \eqref{M1}, then the measure defined by $\mu(A)=\frac{1}{2}\int_A \sin(\theta)d\theta$ is $T$-invariant. 
\end{proposition}

\begin{proof}
Consider the billiard table as an isosceles triangle whose boundary has length $L>1$ and whose base is a straight line segment $\overline{pq}$ of length $1$, which will be identified by the interval $[0,1]$.

The (deterministic) billiard map $F\colon [0,L)\times (0,\pi)\to [0,L)\times (0,\pi)$ preserves the measure $\lambda \times \mu$, where $\lambda$ is the (normalized) Lebesgue measure in $[0,L)$.
In addition, the first return map $G\colon [0,1]\times (0,\pi)\to [0,1]\times (0,\pi)$ with respect to the base $\overline{pq}$ preserves the measure $(\lambda \times \mu)|_{[0,1]\times (0,\pi)}$.

Now, consider $\pi_2\colon [0,L)\times (0,\pi)\to (0,\pi)$ as the projection on the second coordinate.
Given any $\mu$-integrable function $f\colon [0,\pi]\to\Real$, we have:
\begin{align}
 \int f \;d\mu &= \iint f \circ \pi_2 \;d( \lambda \times \mu )\nonumber\\
 &= \iint f \circ \pi_2 \circ G\; d(\lambda \times \mu )\nonumber\\
 &=\int_{[0,\pi]} \int_{[0,1]} f(\pi_2(G(x,\theta)))\; d\lambda(x) d\mu(\theta)\nonumber\\
 &= \int_{[0,\pi]} \int_{[0,1]} f(T_x(\theta))\; d\lambda(x) d\mu(\theta).\label{eq:mu(f)}
\intertext{On the other hand, for any measurable set $A\subset [0,\pi]$, we also have:}
  \sum _{i=1}^4 \int_{[0,\pi]} p_i(\theta) 1_A(T_i(\theta))\; d\mu(\theta) &= \int_{[0,\pi]}K(\theta,A)\; d\mu(\theta)\nonumber\\
  &= \int_{[0,\pi]}\int_{[0,1]} 1_A(T_x(\theta))d\lambda(x)d\mu(\theta).\label{eq:K(theta,A)}
\end{align}

Considering $f$ as the characteristic function $1_A$, then the combination of \eqref{eq:mu(f)} and \eqref{eq:K(theta,A)} implies that $\mu(A) = \int 1_A\; d\mu = \sum _{i=1}^4 \int_{[0,\pi]} p_i(\theta) 1_A(T_i(\theta))\; d\mu(\theta)$.
Therefore, it follows that $\mu$ is $T$-invariant, as in Definition \ref{medida invariante}.
\end{proof}

Let us denote the space of unilateral sequences of four symbols by $\Sigma_4 = \{1,2,3,4 \}^{\mathbb{N}}$.
An element of such space is represented by $\underline{x}=(x_1,x_2,\dots)$.
With this in mind we give some definitions that will be useful later.
For this purpose we let the angle $\theta\in (0,\pi)$ be fixed.

\begin{definition}\label{sigmatheta}
Let $\Sigma_{\theta}\subset\Sigma_4$ be the subspace defined by
\begin{equation*}
\Sigma_\theta :=\{ \underline{x} \in \Sigma\colon p_{x_1}(\theta)p_{x_2}(T_{x_1}(\theta))\dots p_{x_k}\big( T_{x_{k-1}}\circ T_{x_{k-2}}\circ \dots \circ T_{x_1}(\theta)\big) >0, \;\; \forall k \geq 1  \}.
\end{equation*}
In other words, a sequence $\underline{x}$ is in $\Sigma_{\theta}$ if the composition $T_{\underline{x}}^{(k)}(\theta):= T_{x_k} \circ T_{x_{k-1}} \circ \dots \circ T_{x_1}(\theta)$ has positive probability for each $k \in \mathbb{N}$.
\end{definition}

\begin{definition} \label{admissivel}
Given $\underline{x}\in \Sigma_{\theta}$, the related sequence formed by elements $\theta_n:=T_{\underline{x}}^{(n)}(\theta)$ is called an \emph{admissible sequence} for $\theta$ with respect to $\underline{x}$.      
\end{definition}


\begin{definition}\label{def1}
Let $\mathcal{C}(\theta)$ be the set of all possible future images of the angle $\theta$ by the map $T$.
Numerically, we have
\begin{equation*}
\mathcal{C}(\theta) :=\left\{ \theta' \in (0,\pi)\colon T_{\underline{x}}^{(k)}(\theta)=\theta' \text{ for some $k\in \mathbb{N}$ and } \underline{x} \in \Sigma_\theta \right\}.
\end{equation*}
In other words, an angle $\theta'$ is in $\mathcal{C}(\theta)$ if there is an admissible sequence $(\theta_n)_n$ such that $\theta_k = \theta'$ for some $k \in \mathbb{N}$.
\end{definition}

\begin{remark}\label{obs1}
We would like to observe that $\theta \in \mathcal{C}(\theta')$ for every $\theta' \in \mathcal{C}(\theta)$, and hence $\mathcal{C}(\theta) = \mathcal{C}(\theta')$ in this case.
Indeed, that is a consequence of $T_1 \circ T_3 = T_3 \circ T_1 = Id$, $T_2 \circ T_2 = Id$, and $T_4 \circ T_4 = Id$.
Therefore, given $\theta, \theta' \in (0,\pi)$, we conclude that either $\mathcal{C}(\theta) \cap \mathcal{C}(\theta') = \emptyset$ or $\mathcal{C}(\theta) = \mathcal{C}(\theta')$.
\end{remark}

In the following we state a property of the set $\mathcal{C}(\theta)$ extracted from \cite{Vales}. 

\begin{proposition}{\cite[Proposition 3.7 and Proposition 3.8]{Vales}}
Considering $\alpha<\frac{\pi}{6}$, the set $\mathcal{C}(\theta)$ is either
\begin{enumerate}
    \item finite if $\alpha$ is rational with $\pi$; or\label{inv}
    \item countably infinite if $\alpha$ is irrational with $\pi$.\label{infinito}
\end{enumerate}
\end{proposition}

Another property of $\mathcal{C}(\theta)$ is that we can associate a Markov chain with it whose transition probabilities are given by $p_1$, $p_2$, $p_3$ and $p_4$. 
The general case is explained in the following.

Let $(X,\mathcal{A})$ be a measurable space with $X=\{ x_1, x_2, \dots , x_r, \dots \}$ finite or countably infinite.
Consider $\Sigma_X = X^{\mathbb{N}}$ as the space of unilateral sequences with alphabet $X$ and endowed with the $\sigma$-algebra product.
Take $\sigma\colon \Sigma_X \to \Sigma_X$ to be the left shift map, that is, the one defined by $\sigma(x_n)_n = (x_{n+1})_n$.
Given a family of transition probabilities $\{p_{x_ix_j}\colon i,j =1,2,\dots,r, \dots \}$, we define the transition matrix $P=(p_{x_ix_j})_{i,j}$ with $\sum_j p_{x_ix_j} = 1$.
This matrix is also called a \textit{stochastic matrix}.

Note that a measure $\mu$ in $X$ is completely characterized by the values $\mu_i = \mu(x_i)$, with $i=1,2, \dots, r, \dots$.
We say that $\mu$ is a \textit{stationary measure} for the Markov chain if it satisfies $\sum_{i} \mu_i p_{x_ix_j} = \mu_j$ for each $j$.
If the stationary measure $\mu$ is a probability, then we say that $\mu$ is a \textit{stationary distribution} for the Markov chain.

We then define the \textit{Markov measure} with respect to $\mu$ as $\nu (m; x_m, \dots, x_n) = \mu_{x_m} p_{x_m, x_{m+1}} \dots p_{x_{n-1}, x_n}$.
Finally, consider the subset $\Sigma_P\subset\Sigma_X$ of all sequences $(x_n)_n$ such that $p_{x_n, x_{n+1}} > 0$ for all $n \in \mathbb{N}$.
The map $\sigma$ restricted to $(\Sigma_P,\nu)$ is called the \textit{Markov shift map}.

\begin{definition}
The stochastic matrix $P$ is \textit{irreducible} if, for all $x_i, x_j$, there is an $n > 0$ such that $p^n_{x_i, x_j} > 0$.
In other words, $P$ is irreducible if it is possible to move from any state $x_i$ to any state $x_j$ in a certain number $n$ of steps depending on $i$ and $j$.
\end{definition}

\begin{example}\label{ex3.2}
Let $\alpha = \frac{\pi}{8}$ and take any $\theta \in (0,\alpha)$.
In this case, the set $\mathcal{C}(\theta)$ is given by:
\begin{equation*}
\mathcal{C}(\theta) = \{\theta, \theta+2\alpha, \theta+4\alpha, \theta+6\alpha, -\theta+2\alpha, -\theta+4\alpha, -\theta+6\alpha, -\theta+8\alpha \}.
\end{equation*}
Note that the Markov chain associated with this set is irreducible and 2-periodic as one can see in Figure~\ref{fig8}.
\begin{figure}[h]
\centering
\includegraphics[scale=0.2]{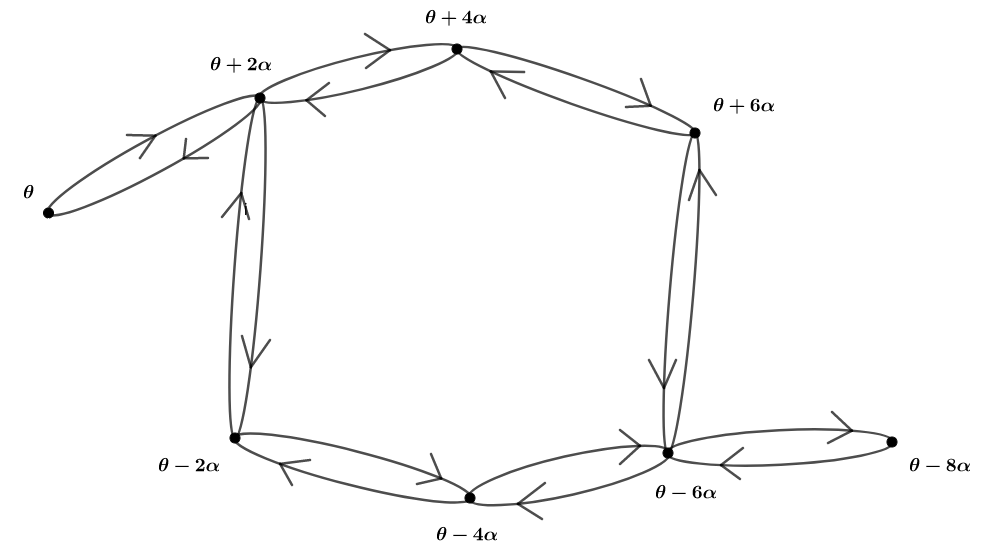}
\caption{Markov chain for $\alpha=\frac{\pi}{8}$ and $\theta \in (0,\alpha).$}\label{fig8}
\end{figure}
\end{example}

\begin{example}
Let $\alpha = \frac{\pi}{7}$ and take any $\theta \in (0,\alpha)$.
In this case, the set $\mathcal{C}(\theta)$ is given by:
\begin{equation*}
\mathcal{C}(\theta)=\{\theta, \theta+2\alpha, \theta+4\alpha, \theta+6\alpha, -\theta+2\alpha, -\theta+4\alpha, -\theta+6\alpha \}.
\end{equation*}
Note that the Markov chain associated with this set is irreducible and aperiodic as one can see in Figure~\ref{fig2}.
\begin{figure}[h]
\centering
\includegraphics[scale=0.2]{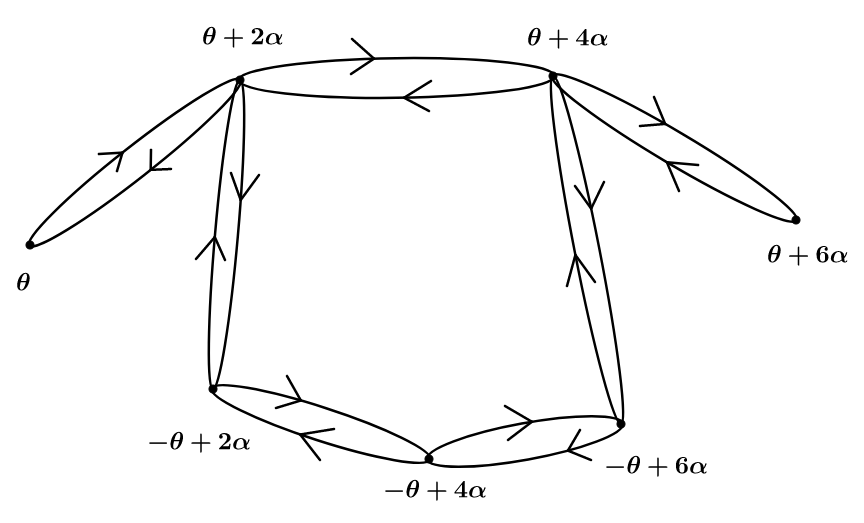}
\caption{Markov chain $\alpha=\frac{\pi}{7}$ and $\theta \in (0,\alpha).$}\label{fig2}
\end{figure}
\end{example}

We end this section with a sufficient condition to the ergodicity of the Markov shift map with respect to the Feres random map $T$.
This result was proven in \cite{Vales}.
Before we state it, recall Definition~\ref{sigmatheta} and, for a fixed angle $\theta \in (0,\pi)$, consider the space $\Sigma = \cup_{\theta' \in \mathcal{C}(\theta)} \Sigma_{\theta'}$.

\begin{proposition}\label{markovrandom}
The Markov shift map in $\Sigma$ is ergodic provided that
\begin{enumerate}
    \item $\alpha$ is rational with $\pi$; or
    \item $\alpha$ is irrational with $\pi$ and the Markov chain is aperiodic.
\end{enumerate}
\end{proposition}

\section{A random billiards map}
Remember that we are considering a curve $\Gamma$ in $\mathbb{M}=\mathbb{R}^2$, $\mathbb{H}^2$, or $\mathbb{S}^2$, whose length is $L>0$.
Also, the corresponding deterministic billiards map $F\colon [0,L)\times (0,\pi)\to [0,L)\times (0,\pi)$ associated with $\Gamma$ is given by $F(S,\theta)=(S_1(S,\theta), \theta_1(S,\theta))$, like discussed in Section~2.
It is known that $F$ preserves the measure $\lambda \times \mu$, where $\lambda$ is the normalized Lebesgue measure with respect to the length of the curve $\Gamma$, and $\mu(A)=\frac{1}{2}\int_A\sin(\theta)d\theta$ as in Proposition~\ref{prop:T-invariance}.

In order to define a random billiards map in the rectangle $M=[0,L)\times (0,\pi)$ we need to consider the following extension of the Feres random map $T$:
\begin{equation*}
\overline{T}(S,\theta)=(S,T_i(\theta))
\end{equation*}
with probability $p_i(\theta)$.

\begin{definition}\label{def:bilhar aleatorio}
The \textit{random billiards map} with respect to $(F,T)$ is the map $\overline{F}\colon M \to M$ defined by
\begin{equation}\label{random}
\overline{F}(S,\theta):= F\circ\overline{T}(S,\theta)=\Big(S_1(S,T_i(\theta)),\theta_1(S,T_i(\theta))\Big)
\end{equation}
with probability $\overline{p}_i(S,\theta):=p_i(\theta)$.
\end{definition}

As we have discussed in the Introduction, random billiards maps defined this way can be interpreted as deterministic billiards maps on a not quite smooth curve whose irregularities are represented by isosceles triangles with base angle $\alpha$.
Since these irregularities are microscopic we can make use of the Gauss-Bonnet Theorem, and hence consider those triangular irregularities as triangles in the Euclidean plane, where the geodesics are straight lines. 

In this paper we are interested in the case when $\Gamma$ is a circle.
Thus the random circular billiards map is given by 
\begin{equation*}
\overline{F}(S,\theta) := F\circ \overline{T}(S,\theta) = (S +l(\gamma(T_i(\theta))) \mod 2\pi h,T_i(\theta))
\end{equation*}
with probability $p_i(\theta)$.

To the next result, recall the definition of evolution of measures, denoted $\nu^{(n)}$, as in \eqref{def:measure evolution}.

\begin{theorem}\label{prop4.4}
For every measure $\nu$ in $[0,\pi]$ absolutely continuous with respect to $\mu$, we have that $\nu^{(n)}\to \mu$ if and only if $\frac{\alpha}{\pi}$ is an irrational number.
\end{theorem}

\begin{proof}
To begin, we would like to observe that Theorem~4.3 of \cite{Vales} holds in $\mathbb{R}^2, \mathbb{H}^2$, and $\mathbb{S}^2$, since the second coordinate of the random circular billiards map $\overline{F}$ is in fact the Feres random map $T$.
Precisely,
\begin{equation*}
\pi_2(\overline{F}(s,\theta))= \pi_2(s + l(\gamma(T_i(\theta))) \mod 2\pi h, T_i(\theta))=T_i(\theta).
\end{equation*}
Therefore, the result follows from Theorem~3.13 of \cite{Vales}.  
\end{proof}

\subsection{An Invariant Measure For The Random Billiards Map}\label{secmedida}

With the terminology introduced in the beginning of this section, let us consider the map $\overline{F}_i(S,\theta)=F\circ \overline{T}_i(S,\theta)$, for each $i=1,2,3,4$.
According to Definition~\ref{medida invariante}, a measure $\nu$ is $\overline{F}$-invariant if
\begin{equation*}
\nu(A)  = \sum_{i=1}^4 \iint _{\overline{F}^{-1}_i(A)}p_i(\theta)d\nu(S,\theta)
\end{equation*}
for every Borel set $A$ of $M$.

\begin{proposition}\label{prop:F-invariant measure}
The measure $\lambda \times \mu$ is $\overline{F}$-invariant either:
\begin{enumerate}
    \item using the usual definition of the random billiards map: $\overline{F}(S,\theta)= F\circ \overline{T}_i(S,\theta)$ with probability $p_i(\theta)$; or
    
    \item using an alternate definition of the random billiards map: $\overline{F}(S,\theta)=\overline{T}_i \circ F(S,\theta)$ with probability $p_i(F(S,\theta))$.
\end{enumerate}
\end{proposition}

\begin{proof}
The proof follows on the same lines as in Propositions 4.5 and 4.6 of \cite{Vales}.
\end{proof}

\section{The random circular billiards map on surfaces of constant curvature}

In this section we will consider the random circular billiards map $\overline{F}(S,\theta) = F\circ \overline{T}(S,\theta) = (S +l(\gamma(T_i(\theta))) \mod 2\pi h,T_i(\theta))$
with probability $p_i(\theta)$ introduced in the previous section.  

\subsection{The Strong Knudsen's Law For the Random Circular Billiards Map}
Recall that Proposition~\ref{prop:F-invariant measure} states that $\overline{F}$ preserves the measure $\lambda \times \mu$, that is, $(\lambda \times \mu)(A\times B) = \sum_{i=1}^4\iint 1_{A\times B}(\overline{F}_i(s,\theta))p_i(\theta)d\lambda (s) d\mu (\theta)$.
That being so, we can define the \textit{transition probability kernel} of the random billiards map $\overline{F}$ as
\begin{equation*}
K((S,\theta), A\times B) := \sum_{i=1}^4 p_i(\theta) 1_{A\times B}(\overline{F}_i(S,\theta)).
\end{equation*}
Once again, this operator describes the evolution of an initial distribution $\nu$ in $(M, \mathcal{B}(M))$ under $\overline{F}$ iteratively as: $\nu^{(0)}:=\nu$ and
\begin{equation*}
\nu^{(n+1)}(A\times B):= \iint K((S,\theta), A\times B) d\nu^{(n)}(S,\theta)
\end{equation*}
for every $A\times B\in\mathcal{B}(M)$ and all $n\ge 1$.

The convergence given in the next proposition is called \textit{Strong Knudsen's Law} for the random circular billiards map (see \cite{lamb} for more details).

\begin{proposition}\label{leiforte}
Let $\frac{\alpha}{\pi}$ be an irrational number, and let $g\colon (0,\pi)\to\Real$ be a continuous function.
If $\nu(S,\theta) := \nu_1(S) \times \nu_2(\theta) = \lambda(S) \times g(\theta)\mu(\theta)$, then
\begin{equation*}
\nu^{(n)}= \nu_1^{(n)}\times\nu_2^{(n)}\longrightarrow (\lambda \times \mu)\text{ setwisely}.
\end{equation*}
\end{proposition}

\begin{proof}
First, observe that $\nu_2:= g\mu$ is absolutely continuous with respect to $\mu$, and so we can make use of Theorem~\ref{prop4.4}.
In addition, we have that
\begin{align*}
\nu^{(1)}(A &\times B) = \iint K((S,\theta), A\times B) d\nu (S,\theta) 
\\
&= \sum_{i=1}^4 \int \int 1_{A \times B}(\overline{F}_i(S,\theta)) p_i(\theta)d(\lambda(S) g(\theta)\mu(\theta))
\\
&=\sum_{i=1}^4 \int p_i(\theta) 1_{B}(T_i(\theta)) \Big(\int 1_{A}(S+l(\gamma(T_i(\theta))) \mod 2\pi h) d\lambda(S)\Big) d(g(\theta)\mu(\theta))
\\
&=\lambda(A) \sum_{i=1}^4 \int 1_B(T_i(\theta))p_i(\theta) d(g(\theta)\mu(\theta)) 
\\
&=\lambda(A)\nu_2^{(1)}(B). 
\end{align*}
We can show by induction that $\nu^{(n)}(A \times B) = \lambda(A)\nu_2^{(n)}(B)$ for all $n\ge 1$.

By the above-mentioned Theorem~\ref{prop4.4}, we already have the setwise convergence $\nu_2^{(n)}\to \mu$ and therefore we conclude that
\begin{equation*}
\nu^{(n)}(A\times B) \longrightarrow \lambda(A)\mu(B)=(\lambda \times \mu)(A \times B)\quad\text{for every }A\times B\in\mathcal{B}(M),
\end{equation*}
as we wanted.
In other words, we just proved the  Strong Knudsen's Law for $\overline{F}$ with respect to a particular family of measures $\nu$.
\end{proof}

The geometric meaning of the above proposition is as follows: given an angle $\alpha$ irrational with $\pi$, the distribution of points and angles after an arbitrarily large number of collisions in the circle is close to the uniform distribution $\lambda \times \mu$.

\subsection{Dense Orbits}
In this section we will show the abundance of dense trajectories of the random circular billiards map $\overline{F}\colon M\to M$ on a surface $\mathbb{M}$ of constant curvature.
We would like to remember that we have fixed $0<\alpha<\frac{\pi}{6}$.

\begin{proposition}\label{densa}
For all $\theta \in (0,\pi)$ such that the arc length $l(\gamma(\theta))$ is an irrational number, there is a dense orbit for the projection of the map $\overline{F}$ in its first coordinate.
\end{proposition}

\begin{proof}
For every $\theta \in (0,\pi-2\alpha)$ notice that $p_1(\theta)p_3(T_1(\theta))>0$ and hence the sequence $\underline{z} = (131313\dots)$ is in $\Sigma_\theta$.
We claim that this sequence generates a dense orbit in the circle.
Indeed, observe first that
\begin{equation*}
\overline{F}_{\underline{z}}^{(2n)}(s,\theta)=(s+2n l(\gamma(\theta+2\alpha)) +2n l(\gamma(\theta)) \mod 2\pi h,\theta),
\end{equation*}
that is, at even times, the random billiards map behaves like a deterministic one of initial angle $2(l(\gamma(\theta + 2\alpha)+ l(\gamma(\theta)))$.
Therefore, if $l(\gamma(\theta))$ is an irrational number, then the orbit is dense in the circle. 
Similarly, if $\theta \in (2\alpha,\pi)$, then we have $p_3(\theta)p_1(T_3(\theta)) >0$.
Thus the sequence $\underline{x}=(313131\dots) \in \Sigma_\theta$ also generates a dense orbit in the circle.
\end{proof}

Let us consider the Markov measure $\nu$ as introduced in Section~3.
Note that we can relate each sequence $\underline{x}\in \cup_{\theta'\in \mathcal{C}(\theta)} \Sigma_{\theta'}$ to a unique sequence $\underline{\theta}\in (\mathcal{C}(\theta))^\mathbb{N}$ and vice versa.
Indeed, given $\theta'\in \mathcal{C}(\theta)$ and
\begin{equation*}
\underline{\theta}=(\theta', \theta'_1, \theta'_2, \cdots) = (\theta', T_{x_1}(\theta'), T_{x_2}\circ T_{x_1}(\theta'), \cdots) \in (\mathcal{C}(\theta))^\mathbb{N},
\end{equation*}
the related sequence in $\cup_{\theta'\in \mathcal{C}(\theta')}\Sigma_{\theta'}$ is $\underline{x}=(x_1, x_2, \cdots)\in \Sigma_{\theta'}$.
On the other hand, given a sequence $\underline{x}=(x_1, x_2, \cdots) \in \Sigma_{\theta'}$,  the related sequence in $(\mathcal{C}(\theta))^\mathbb{N}$ is
\begin{equation*}
\underline{\theta}= (\theta', \theta'_1, \theta'_2, \cdots) = (\theta', T_{x_1}(\theta'), T_{x_2}\circ T_{x_1}(\theta'), \cdots).
\end{equation*}

Thus, for each $\underline{\theta} \in \big(\mathcal{C}(\theta) \big)^\mathbb{N}$, let us consider the $\overline{F}$-trajectory of its related sequence $\underline{x} \in \cup_{\theta' \in \mathcal{C}(\theta)} \Sigma_{\theta'}$, which is denoted by $\overline{F}_{\underline{x}}.$

\begin{theorem}\label{orbitadensa}
For all $(s,\theta) \in M$ and $\nu$-almost every $\underline{\theta}\in\big(\mathcal{C}(\theta)\big)^{\mathbb{N}}$, the trajectory $\overline{F}_{\underline{x}}(s,\theta)$ is dense in the circle $\Gamma$, provided that 
\begin{enumerate}[\qquad 1.]
\item $\alpha$ is rational with $\pi$ and the arc length $l(\gamma(\theta))$ is an irrational number; or
\item $\alpha$ is irrational with $\pi$ and the Markov chain is aperiodic in $\mathcal{C}(\theta)$.
\end{enumerate}
\end{theorem}

\begin{proof}
This will follow the same lines as in Theorem~6.3 of \cite{Vales}.
As one can see in the proof of Proposition~\ref{densa}, it is enough to choose a $\theta'\in \mathcal{C}(\theta)$ such that $p_1(\theta')p_3(T_1(\theta'))>0$.
In this case, the sequence $z=(131313\cdots)\in \Sigma_{\theta'}$ generates a dense trajectory at $\Gamma$.
Therefore, by ergodicity of the Markov shift map (\textit{cf.} Proposition~\ref{markovrandom}), it follows that, given any finite number $N\in\mathbb{N}$, the word `$1313\dots 13$' (with $N$ times the subword `$13$') appears at almost every $\underline{x}\in \Sigma_{\theta'}$.
The same is true for almost every $\underline{x}\in \Sigma_\theta$ and hence the result follows. 
\end{proof}

\subsection{Lyapunov Exponent}

Recall that the (deterministic) circular billiards map in $\mathbb{M}$ given in Equation~\eqref{bilharcircular} is a rotation of $l(\gamma(\theta))$ in the first coordinate, where $l$ represents the arc length between two collisions.
If $\beta(t)$ is a parameterization of the circle, then the arc length between $t_0$ to $t_1$ is
\begin{equation*}
l(\gamma(\theta)) =  \int_{t_0}^{t_1}|\beta'(t)| dt.
\end{equation*}
Thus, the function $l'$ is bounded by $|\beta'|$.

\begin{definition}
Fix some $\underline{x}\in\Sigma_{\theta}$.
The Lyapunov exponent of the random map $\overline{F}$ with respect to $\underline{x}$ at $(s,\theta)$ in the direction $\vec{v}\in\mathbb{M}$ is given by
\begin{equation*}
\lambda_{\underline{x}}(\vec{v}, (s,\theta))=\lim_{n \rightarrow \infty} \frac{1}{n}\log\|D_{(s,\theta)}\overline{F}^{(n)}_{\underline{x}}\cdot\vec{v}\|,
\end{equation*}
when this limit exists and when $D_{(s,\theta)}\overline{F}^{(n)}_{\underline{x}}$ is well defined for all $n\in\mathbb{N}$.
\end{definition}

\begin{proposition}
The Lyapunov exponent of the random circular billiards map in the circle is zero for all $(s,\theta)\in M$, $\underline{x}\in \Sigma_\theta$ and every direction $\vec{v} \in \mathbb{M}$.
\end{proposition}

\begin{proof}
First, notice that, given any $(s,\theta)\in M$ and $\underline{x}\in\Sigma_{\theta}$, the derivative $D_{(s,\theta)}\overline{F}^{(n)}_{\underline{x}}$ is represented by the matrix
\begin{equation*}
D_{(s,\theta)}\overline{F}^{(n)}_{\underline{x}}=\begin{bmatrix}
1 & A_nh \\
0 & B_n
\end{bmatrix},
\end{equation*}
where
\begin{align*}
A_n&= \sum_{k=1}^n l'\Big(\gamma\big(T_{x_k} \circ T_{x_{k-1}}\circ \dots \circ T_{x_1}(\theta))\big)\Big)T'_{x_k}(T_{x_{k-1}}\circ \dots \circ T_{x_1}(\theta))\dots T'_{x_1}(\theta)
\intertext{and}
B_n&=T_{x_k}'(T_{x_{k-1}}\circ \dots \circ T_{x_1}(\theta))T'_{x_{k-1}}(T_{x_{k-2}}\circ \dots \circ T_{x_1}(\theta))\dots T'_{x_2}(T_{x_1(\theta)})T'_{x_1}(\theta).
\end{align*}

Since the maps $T_{x_i}$ are linear, and as $l'$ is bounded by $|\beta'|$, we conclude that $|A_n|\leq |\beta'|n$.
We also have $B_n \in \{-1,1\}$.
Thus, taking any direction $\vec{v}=(v_1,v_2)\in \mathbb{M}$, we obtain that 
\begin{equation*}
\frac{1}{n}\log \|D_{(s,\theta)} \overline{F}_{\underline{x}}^{(n)} \cdot \vec{v} \| =\frac{1}{n} \log \| (v_1+v_2A_nh, v_2B_n) \|.
\end{equation*}
Therefore, as $A_n$ and $B_n$ are bounded, it follows that
\begin{equation*}
\lambda_{\underline{x}}(\vec{v},(s,\theta))=\lim_{n\to \infty} \frac{1}{n} \log \|D_{(s,\theta)}\overline{F}_{\underline{x}}^{(n)} \cdot \vec{v}\| = 0.
\end{equation*} 
\end{proof}

\section{Mixing and Pseudo integrability}

\subsection{Mixing}

We begin this subsection constructing a skew-product related to the random map $\overline{F}$ that will be useful in the definition of mixing.
In fact, we need the following elements extracted from \cite{Quas} and \cite{lamb}: considering $X=[0,\pi]\times [0,1]$ we can define the sets $J_i=\{(\theta,x) \in X\colon \sum_{k<i}p_k(\theta) < x < \sum_{k\leq i} p_k(\theta)\}$, for  $i=1,2,3,4$.
Note that $\{J_1,J_2,J_3,J_4\}$ is a partition of $X$ (see Figure~\ref{decomposicaoJ} below).
\begin{figure}[h]
    \centering
    \includegraphics[scale=0.5]{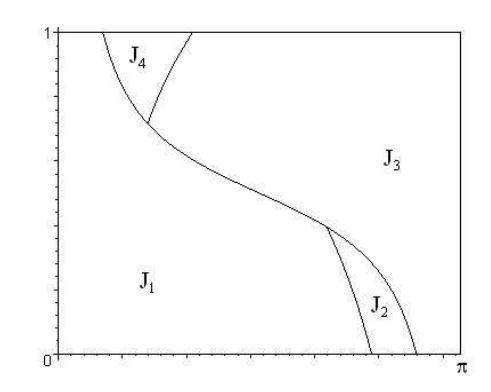}
    \caption{A partition of $X$ composed by the sets $J_i$.}
    \label{decomposicaoJ}
\end{figure}
Besides, for each $i=1,2,3,4$, let $\varphi_i\colon X\to[0,1]$ be given by $\varphi_i(\theta,x) = \frac{1}{p_i(\theta)}\big(x-\sum_{k=1}^{i-1} p_k(\theta)\big)$.
Observe that each map $\varphi_i$ is well defined.
Indeed, if $(\theta,x) \in J_i$, then $p_i(\theta) > 0$ necessarily.

With these we define the skew-product:
\begin{equation*}
\begin{matrix}
S\colon & [0,1]\times[0,L]\times[0,\pi] & \longrightarrow & [0,1]\times[0,L]\times[0,\pi]\\
& (x,s,\theta) & \longmapsto & (\varphi_i(\theta,x),s+2T_i(\theta)\mod{2\pi h},T_i(\theta))
\end{matrix}
\end{equation*}
if $(\theta,x)\in J_i$, for $i=1,2,3,4$.
This skew-product has an invariant measure given by $\eta\times\lambda\times\mu$, where $\eta$ and $\lambda$ are the normalized Lebesgue measures on the interval $[0,1]$ and $[0,L]$, respectively.
Recall that $\mu$ is as defined in Proposition~\ref{prop:T-invariance}, that is, such that $d\mu=\frac{1}{2}\sin(\theta)d\theta$.

Following \cite{Quas} and \cite{morita1988deterministic}, we can define: 
\begin{definition}
The random circular billiards map $\overline{F}$ is \textit{mixing} (or \textit{strongly mixing}) if the skew-product $S$ is mixing (or strongly mixing).
\end{definition}

We would like to remark that this definition can be extended to all random billiards map.
For more details, see \cite{Quas}.

In the deterministic case, the billiards map on a surface of constant curvature cannot be ergodic due to its integrability.
However, we will show that, in the random case, the map is mixing whenever the parameter $\alpha$ is irrational with $\pi$.
This is one of the features that strongly differs the deterministic case from the random one.

On the other hand, when the parameter $\alpha$ is rational with $\pi$, we are able to find invariant regions whose measure is different from 0 and 1, so it cannot be ergodic.
In order to see this we can use, for example, \cite[Proposition~3.10]{Vales}.

\begin{theorem}[Main Theorem]
Considering a surface of constant curvature, the random circular billiard map is mixing if and only if $\frac{\alpha}{\pi}$ is an irrational number.
\end{theorem}

\begin{proof}
Denote the random circular billiards map by $\overline{F}$ and its associated skew-product by $S(x,\tilde{\theta}):= S(x,s,\theta)$ as in the beginning of this subsection.

First, let us assume that $\alpha$ is irrational with $\pi$.
Using Proposition~\ref{leiforte}, if $\nu:= \lambda \times \nu_2 \ll \lambda \times \mu$, then $\nu^{(n)}(E)\to (\lambda \times \mu)(E)$ for every measurable set $E\in \mathcal{B}([0,L]\times [0,\pi])$.
In addition, \cite[Theorem~4 and Corollary~5]{feres} (with $X= [0,L] \times [0,\pi]$ and $\nu = \lambda \times \nu_2$, where $\nu_2 \ll \mu$) state that $\nu^{(n)}(E) = \int_{[0,1] \times X} U_S^n (1_{\pi_2^{-1}(E)})d(\eta \times \nu)$.

Since $\nu=\lambda \times \nu_2 \ll \lambda \times \mu$, let $f \in L^1([0,L]\times [0,\pi],\lambda \times \mu)$ be the Radon-Nikodym derivative.
Thus, $\pi^*_2(f) = \frac{d(\eta \times \lambda \times \nu_2)}{d(\eta \times \lambda \times \mu)}$, where $\pi^*_2(f)(x,\tilde{\theta}) = f(\pi_2(x,\tilde{\theta}))$ for every $\tilde{\theta}=(s,\theta)\in [0,L]\times[0,\pi]$.
It then follows that:
\begin{align*}
\int \pi^*_2(f)&U_S^n(1_{\pi_2^{-1}(E)})d(\eta \times \lambda \times \mu) = \int_{[0,1] \times X} U_S^n(1_{\pi_2^{-1}(E)})d(\eta \times \lambda \times \nu_2 ) \\ & = \nu^{(n)}(E)\to(\lambda \times \mu) (E) = \int \pi^*_2(f) d(\eta \times \lambda \times \mu) \int 1_{\pi_2^{-1}(E)} d(\eta \times \lambda \times \mu)
\end{align*}
for every measurable set $E\in\mathcal{B}([0,L]\times[0,\pi])$.

Therefore, \cite[Proposition~4.4.1(b)]{lasota} guarantees that $S$ is mixing.
This means by definition that $\overline{F}$ is mixing.
To conclude the reverse statement just notice that the results mentioned in this proof hold only if $\frac{\alpha}{\pi}$ is an irrational number. 
\end{proof}

\subsection{Pseudo Integrability}
We start by remembering that a deterministic dynamical system $F\colon M\to M$ in a $2$-dimensional manifold is said to be \textit{integrable} if it admits a non-constant smooth and invariant function $G\colon M\to\Real$, called the motion constant.
Symbolically, this means that $G(F(x, y)) = G(x, y)$ for every $(x,y)\in M$.
Furthermore, if $M$ is foliated by $1$-dimensional submanifolds (\textit{i.e.}, curves), then the system is called \textit{fully integrable}.
It is well-known that the (deterministic) circular billiards map is fully integrable and its phase space is covered by invariant curves.

On the other hand, we say (informally) that a random billiards map is \textit{pseudo integrable} if its phase space is foliated by curves, which are invariant as a set.
The precise meaning is given in Definition~\ref{def:pseudo integravel}.
As an example, for each chosen sequence in the circular case, we would have a set of curves (segments of straight lines) in the phase space that are invariant.

\begin{definition}
A sequence $\underline{x} \in \Sigma$ is \textit{almost admissible} for $\theta \in [0,\pi]$ if $ T_{x_k}\circ \dots\circ T_{x_2}\circ T_{x_1}(\theta)$ occurs with probability $p_{x_k}(T_{x_{k-1}}\circ \dots \circ T_{x_1}(\theta) )\cdots p_{x_2}(T_{x_1}(\theta))p_{x_1}(\theta) \geq 0$ and $T_{x_1}(\theta), T_{x_2}\circ T_{x_1}(\theta), \dots, T_{x_k}\circ \dots \circ T_{x_1}(\theta) \in [0,\pi]$ for all $k\in \mathbb{N}.$
\end{definition}

The above definition accepts all zero-probability compositions whose images belong to the range $[0,\pi]$ (compare with Definition~\ref{admissivel}).

\begin{lemma}\label{lema}
Fix some $\alpha<\frac{\pi}{6}$.
Given $\theta \in [0,\pi]$, there is an almost admissible sequence $\underline{x}$ such that $T_{x_k}\circ \dots\circ T_{x_2}\circ T_{x_1}(\theta) \in [0,\alpha]$ for some $k \in \mathbb{N}$.
\end{lemma}

\begin{proof}
First, suppose that $\theta \in [\alpha, 2\alpha]$.
This implies that $T_4(\theta)\in [2\alpha, 3\alpha]$ and hence $T_3\circ T_4(\theta)\in [0, \alpha]$.
In this case just take any almost admissible sequence $\underline{x}$ with $x_1=4$ and $x_2=3$.

Now suppose that $\theta\in[0,\pi]\setminus[\alpha,2\alpha]$.
Notice that there is a maximum number $k\in\mathbb{N}$ of times that we can apply $T_3$ consecutively.
Thus either $T_3^k(\theta) \in [0,\alpha]$ or $T_3^k(\theta) \in [\alpha, 2\alpha]$.
The former case are done, and the latter brings us to our previous assumption, that is, $T_3 \circ T_4 \circ T_3^k(\theta) \in [0,\alpha]$.
In both cases we can take an almost admissible sequence satisfying the statement.
\end{proof}

In order to define the notion of \textit{pseudo integrability} of random billiards maps, let us establish the following equivalence relation.

\begin{definition}
We say that $\theta,\theta'\in[0,\pi]$ are related (denoted by $\theta \sim \theta'$) if there is an almost admissible sequence $\underline{x}$ such that $\theta'= T_{x_k}\circ \dots \circ T_{x_1}(\theta)$ for some $k \in \mathbb{N}$.
\end{definition}

This leads us to the following immediate corollary of Lemma~\ref{lema}.

\begin{corollary}\label{lemma5.2}
For $\alpha < \frac{\pi}{6},$ we have that $\dfrac{[0, \pi]}{\sim} \equiv [0, \alpha]$.
\end{corollary}

We can now define pseudo integrability.
Consider from now on the rectangle $M=[0,L)\times[0,\alpha]$ for a fixed $\alpha<\frac{\pi}{6}$.

\begin{definition}\label{def:pseudo integravel}
A random billiards map $\overline{F}\colon M\to M$ is called \textit{pseudo integrable} if there is a motion constant in the quotient phase space, that is, if one can take a function $G\colon M\to\Real$ such that $G\circ\overline{F} \sim G$. Furthermore, if $M$ is foliated by invariant curves, then the random map is called \textit{fully pseudo integrable}.
\end{definition}

\begin{proposition}
The random circular billiards map $\overline{F}\colon M\to M$, as defined in \ref{def:bilhar aleatorio}, is fully pseudo integrable. 
\end{proposition}

\begin{proof}
Fix any $\theta\in[0,\alpha]$.
By Remark~\ref{obs1}, the set $\mathcal{C}(\theta)$ is invariant by the Feres map $T$.
Let $G(s,\theta) = \theta$ and take an almost admissible sequence $\underline{x}$ for $\theta$.
Then, $G(\overline{F}_{\underline{x}}(s, \theta))= T_i(\theta)$ with probability $p_i(\theta)$.
Finally, note that $\theta \sim T_i(\theta)$, that is, $G(\overline{F}_{\underline{x}}(s, \theta)) \sim \theta = G(s,\theta)$. See Figure \ref{pseudointegrabilidade}.
\end{proof}

\begin{figure}[ht]
    \centering
    \includegraphics[scale=0.055]{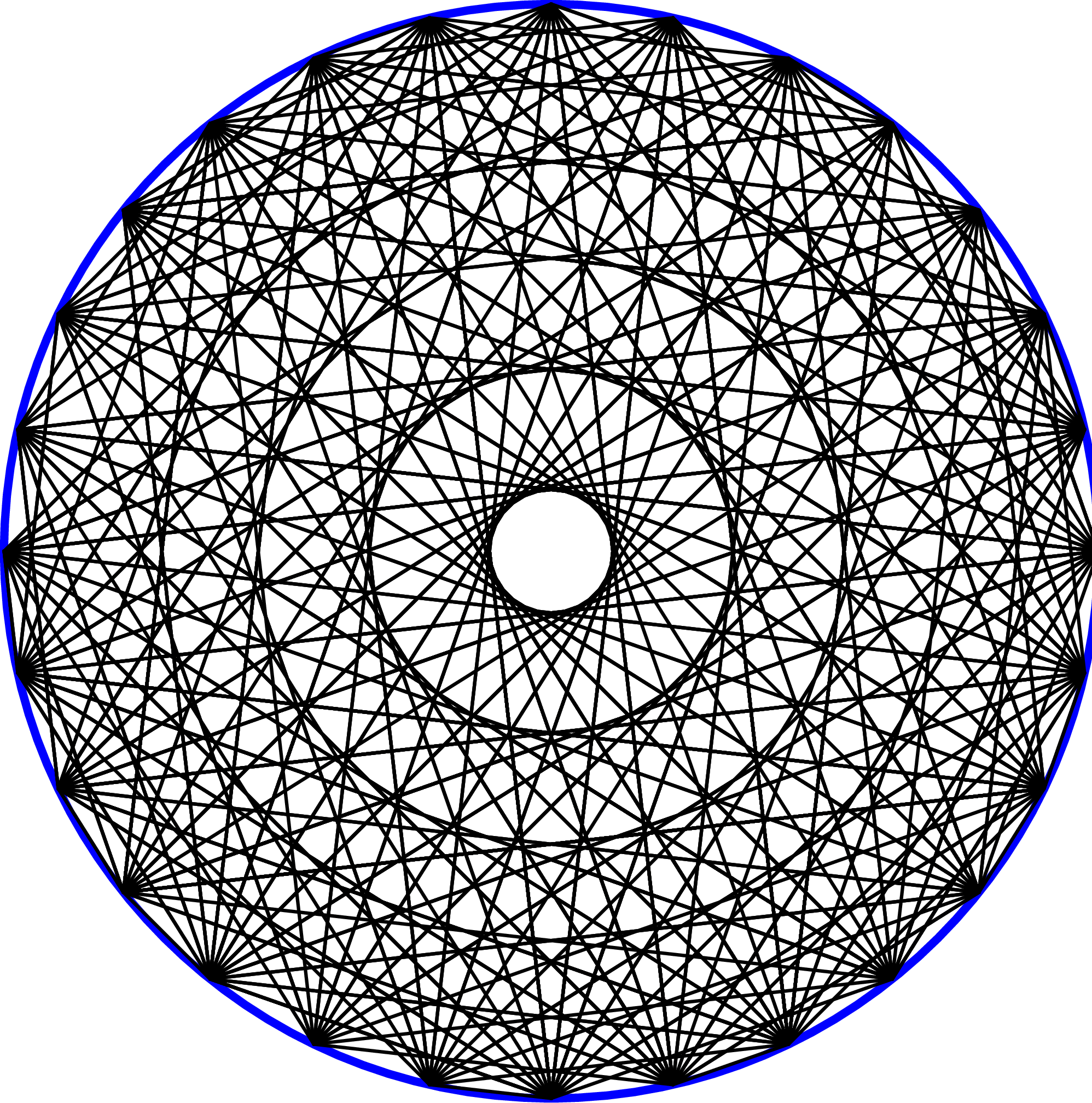}
    \includegraphics[scale=0.1]{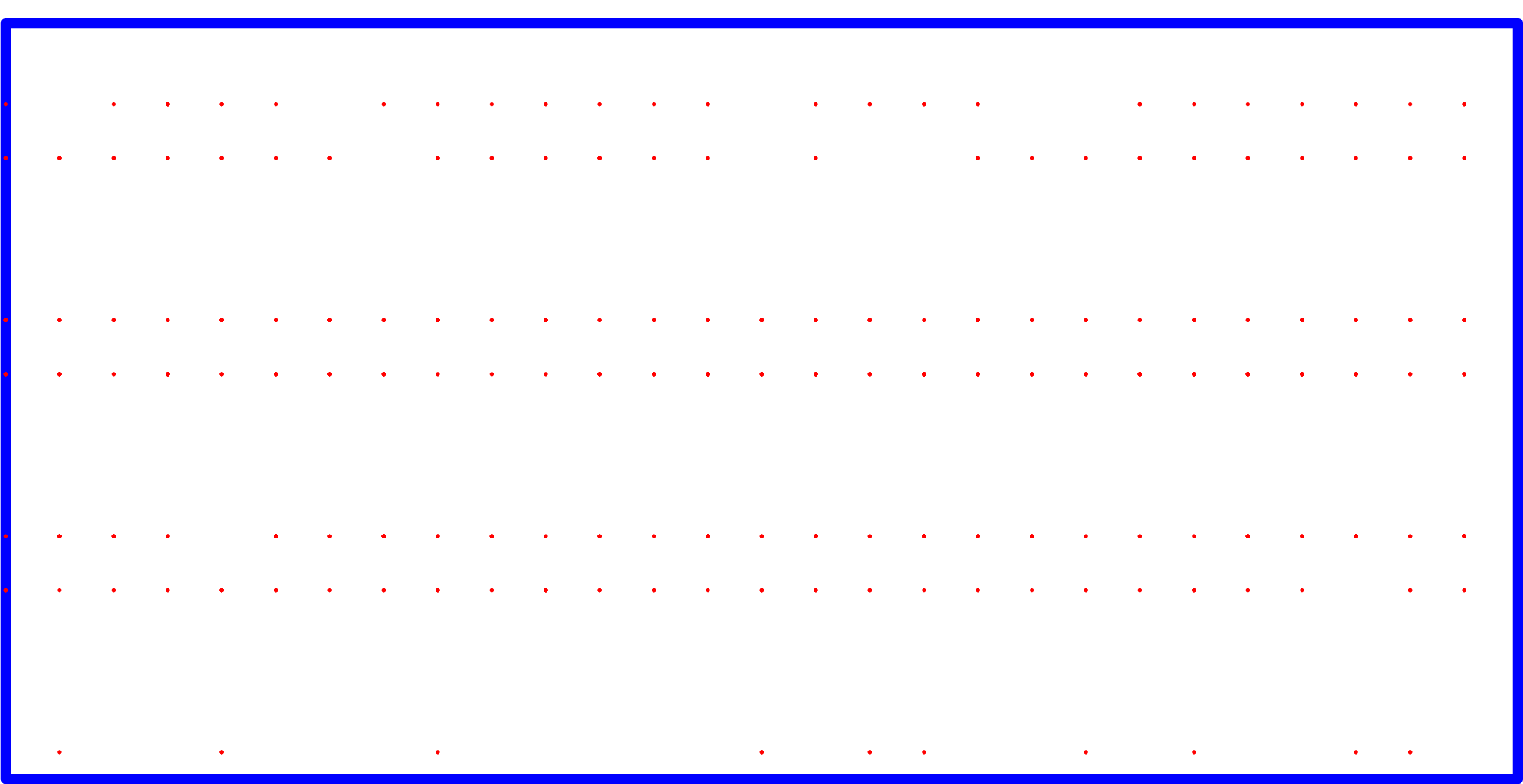}
    \caption{Numerical evidence for pseudo integrability.}
    \label{pseudointegrabilidade}
\end{figure}


\begin{remark}
Notice that, in general curves, the angles of reflection depend on the position, which does not happen in the circular case.
Therefore, even in the quotient space with respect to $\mathcal{C}(\theta)$, is not straightforward to check for pseudo integrability.
\end{remark}

\section{Acknowledgment}
The author would like to thank the referees for all corrections and suggestions.
The author also thanks Cláudia P. Ferreira - IFMG, Luís F. S. Amaral - UFMG, and Rafael C. Pereira - UFMG for their contributions and discussions. 

\bibliographystyle{siam}
\bibliography{bib}

\begin{thebibliography}{10}

\bibitem{Quas}
{\sc W.~Bahsoun, C.~Bose, and A.~Quas}, {\em Deterministic representation for
  position dependent random maps}, Discrete and Continuous Dynamical Systems,
  22 (2008), pp.~529--540.

\bibitem{markarian}
{\sc N.~Chernov and R.~Markarian}, {\em Chaotic billiards}, no.~127, American
  Mathematical Soc., 2006.

\bibitem{lamb}
{\sc K.~Dingle, J.~S.~W. Lamb, and J.-A. L{\'{a}}zaro-Cam{\'{\i}}}, {\em
  Knudsen's law and random billiards in irrational triangles}, Nonlinearity, 26
  (2012), pp.~369--388.

\bibitem{luciano}
{\sc L.~C. dos Santos and S.~Pinto-de Carvalho}, {\em Periodic orbits of oval
  billiards on surfaces of constant curvature}, Dynamical Systems, 32 (2017),
  pp.~283--294.

\bibitem{feres}
{\sc R.~Feres}, {\em Random walks derived from billiards}, Dynamics, ergodic
  theory, and geometry, 54 (2007), pp.~179--222.

\bibitem{gutkin1999hyperbolic}
{\sc B.~Gutkin, U.~Smilansky, and E.~Gutkin}, {\em Hyperbolic billiards on
  surfaces of constant curvature}, Communications in mathematical physics, 208
  (1999), pp.~65--90.

\bibitem{lasota}
{\sc A.~Lasota and M.~C. Mackey}, {\em Probabilistic properties of
  deterministic systems}, Cambridge university press, 2008.

\bibitem{morita1988deterministic}
{\sc T.~Morita}, {\em Deterministic version lemmas in ergodic theory of random
  dynamical systems}, Hiroshima mathematical journal, 18 (1988), pp.~15--29.

\bibitem{stillwell2012geometry}
{\sc J.~Stillwell}, {\em Geometry of surfaces}, Springer Science \& Business
  Media, 2012.

\bibitem{Vales}
{\sc T.~Vales and S.~P. de~Carvalho}, {\em A random billiard map in the
  circle}, Dynamical Systems, 37 (2) (2022), pp.~333--353.

\end{thebibliography}
\end{document}